\pdfoutput=1  

\documentclass{amsart}
\usepackage{amssymb}

\usepackage{graphicx}

\usepackage{caption} 
\newtheorem{theorem}{Theorem}[section]
\newtheorem{lemma}[theorem]{Lemma}

\newtheorem{proposition}[theorem]{Proposition}
\newtheorem{claim}[theorem]{Claim}

\newtheorem{corollary}[theorem]{Corollary}
\newtheorem{question}[theorem]{Question}

\theoremstyle{definition}
\newtheorem{example}[theorem]{Example}

\theoremstyle{remark}
\newtheorem{remark}[theorem]{Remark}

\numberwithin{equation}{section}

\newcommand{\FL}{{\rm FL}}

\begin{document}

\title[Hyperbolic knots with large torsion order]{Hyperbolic knots with arbitrarily large torsion order
in knot Floer homology}



\author[K. Himeno]{Keisuke Himeno}
\address{Graduate School of Advanced Science and Engineering, Hiroshima University,
1-3-1 Kagamiyama, Higashi-hiroshima, 7398526, Japan}
\email{himeno-keisuke@hiroshima-u.ac.jp}
\thanks{The first author was supported by JST SPRING, Grant Number JPMJSP2132. }
\author[M. Teragaito]{Masakazu Teragaito}
\address{Department of Mathematics and Mathematics Education, Hiroshima University,
1-1-1 Kagamiyama, Higashi-hiroshima 7398524, Japan.}
\email{teragai@hiroshima-u.ac.jp}
\thanks{The second author has been partially supported by JSPS KAKENHI Grant Number JP20K03587.}


\date{\today}


\commby{}

\begin{abstract}
In knot Floer homology, there are two types of torsion order. 
One  is the minimal power of the action of the variable $U$ to annihilate the
$\mathbb{F}_2[U]$-torsion submodule of the minus version of knot Floer homology $\mathrm{HFK}^-(K)$.
This is introduced by Juh\'{a}sz, Miller and Zemke, and denoted by $\mathrm{Ord}(K)$.
The other, $\mathrm{Ord}'(K)$, introduced by Gong and Marengon,  is similarly defined for the $\mathbb{F}_2[U]$-torsion submodule of  the unoriented knot Floer homology $\mathrm{HFK}'(K)$.

For both torsion orders, it is known that arbitrarily large values are realized by torus knots.
In this paper, we prove that they can be realized by hyperbolic knots, most of which are twisted torus knots.
Two torsion orders are argued in a unified way by using the Upsilon torsion function introduced by Allen and Livingston.
We also give the first  infinite family of hyperbolic  knots which shares a common Upsilon torsion function.

\end{abstract}


\maketitle


\section{Introduction}\label{sec:intro}

There are two types of torsion order in knot Floer homology.
The first one is introduced by Juh\'{a}sz, Miller and Zemke \cite{JMZ}.
Recall that the minus version of knot Floer homology $\mathrm{HKF}^-(K)$ is a finitely generated module over
the polynomial ring $\mathbb{F}_2[U]$.
Let us denote $\mathrm{Tor}(\mathrm{HFK}^-(K))$ its $\mathbb{F}_2[U]$-torsion submodule.
Then the torsion order of a knot $K$ is defined as
\[
\mathrm{Ord}(K)=\min \{ k\ge 0 \mid U^k\cdot \mathrm{Tor}(\mathrm{HFK}^-(K))=0  \} \in \mathbb{N}\cup \{0\}.
\]
Of course, for the unknot $O$, $\mathrm{Ord}(O)=0$.
Since knot Floer homology detects the unknot \cite{OS0}, $\mathrm{Ord}(K)\ge 1$ when $K$ is non-trivial.
For example,  for the torus knot $T(p,q)$ with $1<p<q$, $\mathrm{Ord}(T(p,q))=p-1$ \cite{JMZ}.
Hence arbitrarily large values of torsion order can be realized by torus knots.
There are  several applications for knot cobordisms. See also \cite{HKP}.

The second is similarly defined  in \cite{GM} by using the torsion submodule of Ozsv\'{a}th, Stipsicz and Szab\'{o}'s unoriented knot Floer homology $\mathrm{HFK}'(K)$, 
which is also a module over $\mathbb{F}_2[U]$ (\cite{OSS}),
instead of $\mathrm{HFK}^-(K)$.
Hence
\[
\mathrm{Ord}'(K)=\min \{ k\ge 0 \mid U^k\cdot \mathrm{Tor}(\mathrm{HFK}'(K))=0  \} \in \mathbb{N}\cup \{0\}.
\]
Again, $\mathrm{Ord}'(K)=0$ if and only if $K$ is trivial.
(For, $\mathrm{HFK}'(O)=\mathbb{F}_2[U]$, which is torsion-free \cite[Corollary 2.15]{OSS}.
Conversely,
if $\mathrm{HFK}'(K)$ is torsion-free, then $\mathrm{HFK}'(K)=\mathbb{F}_2[U]=
\mathrm{HFK}'(O)$  \cite[Proposition 3.5]{OSS}.
So,  the unoriented knot Floer complexes $\mathrm{CFK}'(K)$ and  $\mathrm{CFK}'(O)$ 
share the same homology, which
implies chain homotopy equivalence between them \cite[Proposition A.8.1]{OSS2}.
Since setting $U=0$ reduces the complex into the hat version of knot Floer complex
\cite[Proposition 2.4]{OSS}, we have
$\widehat{\mathrm{HFK}}(K)\cong \widehat{\mathrm{HFK}}(O)$ by \cite[Proposition A.3.5]{OSS2}.
This implies $K=O$.)

Gong and Marengon \cite[Lemma 7.1]{GM} verify $\mathrm{Ord}'(T(p,p+1))=\lfloor \frac{p}{2} \rfloor$.
Hence arbitrarily large values of this torsion order can  be realized by torus knots, again.

As shown in \cite{AL}, two types of torsion order can be unified
in terms of the Upsilon torsion function $\Upsilon^{\mathrm{Tor}}_K(t)$,
 which is a piecewise linear continuous function defined on the interval $[0,2]$.
The derivative of $\Upsilon^{\mathrm{Tor}}_K(t)$ near $0$ equals to $\mathrm{Ord}(K)$, and
$\Upsilon^{\mathrm{Tor}}_K(1)=\mathrm{Ord}'(K)$.
We remark that the Upsilon torsion function and two types of torsion order are not concordance invariats.


The main purpose of this paper is to confirm that arbitrarily large values of 
these two types of torsion order can be realized by hyperbolic knots.
Except a few small values, we make use of twisted torus knots.

\begin{theorem}\label{thm:main}
Let $K$ be a twisted torus knot $T(p,kp+1;2,1)$ with $k\ge 1$.
\begin{itemize}
\item[(1)]  If $p\ge 2$, then $\mathrm{Ord}(K)=p-1$.
\item[(2)] If $p\ge 4$, then $\mathrm{Ord}'(K)=\lfloor\frac{p-2}{2}\rfloor$.
\end{itemize}
\end{theorem}


Unfortunately, a twisted torus knot $T(p,kp+1;2,1)$ is not hyperbolic when $p<5$ (see Proposition \ref{prop:hyp}).
However, an additional argument gives the following.

\begin{corollary}\label{cor:main}
Let $N\ge 1$ be a positive integer.
Then there exist infinitely many hyperbolic knots $K_1$ and $K_2$ with $\mathrm{Ord}(K_1)=N$ and $\mathrm{Ord}'(K_2)=N$.
\end{corollary}

\begin{corollary}\label{cor:main2}
There exist infinitely many hyperbolic knots that share the same Upsilon torsion function.
\end{corollary}

We pose a simple question.

\begin{question}
Let $M$ and $N$ be positive integers.
Does there exist a knot $K$ with $(\mathrm{Ord}(K), \mathrm{Ord}'(K))=(M,N)$?
\end{question}

\begin{remark}
For two types of torsion order, 
the original symbols are  $\mathrm{Ord}_v(K)$ and  $\mathrm{Ord}_U(K)$ (see \cite{JMZ,GM}).
\end{remark}





\section{Twisted torus knots}

A twisted torus knot is obtained from a torus knot of type $(p,q)$ by twisting  $r$ adjacent strands by $s$ full twists.
The resulting knot is denoted by $T(p,q;r,s)$ as in literatures \cite{L0,L1,L2,L3}.

Throughout this section, let $K$ be the twisted torus knot $T(p,kp+1;2,1)$ with $p\ge 2, k\ge 1$.
Clearly, if $p=2$, then $T(2,2k+1;2,1)=T(2,2k+3)$.
Also, Lee \cite{L2,L3} shows that $T(3,3k+1;2,1)=T(3,3k+2)$, and
$T(4,4k+1;2,1)$ is the $(2,8k+3)$-cable of $T(2,2k+1)$.
We will show later  that $T(p,kp+1;2,1)$ is hyperbolic if $p\ge 5$ (Proposition \ref{prop:hyp}).
Since these knots are the closure of a positive braid, it is fibered by \cite{S}.
In particular, the Seifert algorithm on a positive braided diagram gives a fiber, which is a minimal genus Seifert surface.
Thus we know that it has genus $(kp^2-kp+2)/2$.
Hence $K$ is non-trivial.

\begin{lemma}\label{lem:tunnel}
$K$ is an L--space knot.
\end{lemma}

\begin{proof}
This follows from \cite{V}.
\end{proof}



\begin{lemma}\label{lem:alex}
The Alexander polynomial $\Delta_K(t)$ of $K$ is given by
\[
\Delta_K(t)=
\begin{cases}
\begin{aligned}
1&+\sum_{i=1}^k(-t +t^{p})t^{(i-1)p}\\
&+\sum_{i=1}^{p-3} \sum_{j=1}^{k}(   -t^{ikp+1}+t^{ikp+2}-t^{ikp+2+i}+t^{(ik+1)p}  ) t^{(j-1)p}\\
&+\sum_{i=1}^{k+1}(      -t^{kp(p-2)+1} + t^{kp(p-2)+2}        )t^{(i-1)p},
\end{aligned}
 & \text{ if $p\ge 3$},\\
1-t+t^2-\dots +t^{2k+2},  &\text{if $p=2$.}
\end{cases}
\]
\end{lemma}

\begin{proof}
When $p=2$, it is well known that
\[
\Delta_K(t)=\frac{ (1-t) (1-t^{2(2k+3)}) } { (1-t^2) (1-t^{2k+3} ) }=1-t+t^2-\dots +t^{2k+2},
\]
since $K=T(2,2k+3)$ as mentioned before.

Assume $p\ge 3$.
The conclusion essentially follows from \cite{Mo}.
In his notation, our knot $K$ is $\Delta(p,kp+1,2)$ with $r=p-1$.
Hence
\begin{align*}
\Delta_K(t) &=\frac{1-t} {(1-t^p)(1-t^{kp+1}) }\cdot  \\
& \qquad ( 1- (1-t) ( t^{(p-1)(kp+1)+1} + t^{kp+1} ) -t^{p(kp+1)+2}   ).
\end{align*}

The second factor is changed as
\[
\begin{split}
 1 &- (1-t) ( t^{(p-1)(kp+1)+1} + t^{kp+1} ) -t^{p(kp+1)+2} =
 1- t^{(p-1)(kp+1)+1}- t^{kp+1}\\
 &  +  t^{(p-1)(kp+1)+2}  + t^{kp+2} -t^{p(kp+1)+2} = (1-t^{kp+1})+\\
 &+t^{kp+2}(1-t^{ (kp+1)(p-2)})+ t^{(p-1)(kp+1)+2}(1-t^{kp+1}).
\end{split}
\]
Thus
\[
\Delta_K(t)=\frac{1-t}{1-t^p}\cdot ( 1 +t^{kp+2} \sum_{i=0}^{p-3} t^{i(kp+1)}+t^{(p-1)(kp+1)+2} ) .
\]

We set 
\begin{align*}
A&=\sum_{i=1}^k(-t +t^{p})t^{(i-1)p}, \\
B&=\sum_{i=1}^{p-3} \sum_{j=1}^{k}(   -t^{ikp+1}+t^{ikp+2}-t^{ikp+2+i}+t^{(ik+1)p}  ) t^{(j-1)p}, \\
C&=\sum_{i=1}^{k+1}(      -t^{kp(p-2)+1} + t^{kp(p-2)+2}        )t^{(i-1)p}.
\end{align*}
Then it is straightforward to calculate
\begin{align*}
(1-t^p)A&=  -t+t^p+t^{kp+1}-t^{(k+1)p},\\
(1-t^p)B&=  -t^{kp+1}+t^{(k+1)p}+(1-t) \sum_{i=0}^{p-3} t^{(i+1)kp+i+2}+t^{kp(p-2)+1}-t^{kp(p-2)+2}, \\
(1-t^p)C&= -t^{kp(p-2)+1} + t^{kp(p-2)+2} + t^{kp(p-2)+1+(k+1)p} - t^{kp(p-2)+2+(k+1)p}.
\end{align*}
Hence
\[
\begin{split}
(1-t^p)(1+A+B+C)&=1-t +(1-t) \sum_{i=0}^{p-3} t^{(i+1)kp+i+2} \\
&\qquad  +(1-t) t^{kp(p-2)+1+(k+1)p}\\
&=(1-t)( 1+ t^{kp+2}\sum_{i=0}^{p-3}t^{i(kp+1)}+t^{ (p-1)(kp+1)+2  }).
\end{split}
\]
This shows that $\Delta_K(t)=1+A+B+C$ as desired.
\end{proof}

\begin{corollary}\label{cor:gap}
The gaps of the exponents of the Alexander polynomial of $K$ are
\[
(1,p-1)^k,(1,1,1,p-3)^k,(1,1,2,p-4)^k,\dots, (1,1,p-3,1)^k,1,1,(p-1,1)^k\]
%
if $p\ge 3$, and
$1^{2k+2}$ if $p=2$.
Here, the power indicates the repetition.
(We remark that  the above sequence is $(1,2)^k,1,1,(2,1)^k$ when $p=3$.)
\end{corollary}

To prove that our twisted torus knot $K=T(p,kp+1;2,1)$ is hyperbolic when $p\ge 5$,
we give a more general result by using \cite{I}.
A knot $k$ is called a \textit{fully positive braid knot\/} 
if it is the closure of a positive braid which contains at least one full twist.

\begin{proposition}\label{prop:full}
Let $k$ be a fully positive braid knot.
If $k$ is a tunnel number one, satellite knot, then $k$ is a cable knot.
\end{proposition}

\begin{proof}
By \cite{MS}, $k$ has a torus knot $T(r,s)$ as a companion.  We may assume that $1<r<s$.
Then Theorem 1.2 of \cite{I} claims that
the pattern $P$ is represented by  a positive braid in a solid torus.

Let us recall the construction of \cite{MS}.
Starting from a $2$-bridge link $K_1\cup K_2$,
consider the solid torus $E(K_2)$ containing $K_1$.
Remark that $K_1$ and $K_2$ are unknotted.
For the companion knot $T(r,s)$,
consider a homeomorphism from $E(K_2)$ to the tubular neighborhood of $T(r,s)$,
which sends the preferred longitude of $E(K_2)$ to the regular fiber of the Seifert fibration in the exterior of $T(r,s)$.
Hence our pattern knot $P$, which is defined under preserving preferred longitudes, 
 is obtained from $K_1$ by adding $rs$ positive full twists to $E(K_2)$.
Since $K_1$ is unknotted, we can set the pattern $P$ as the closure of a positive braid
\[
\sigma_{i(1)} \sigma_{i(2)} \dots \sigma_{i(n-1)} (\sigma_1 \sigma_2 \dots \sigma_{n-1})^{nrs}
\]
for some $n\ge 2$, where $\{i(1), i(2), \dots, i(n-1)\}=\{1,2,\dots, n-1\}$.
(If the initial part before $rs$ full twists contains more than $n-1$ generators,
then the Seifert algorithm gives a fiber surface of the closure $K_1$, which has positive genus.)

For two braids $\beta_1$ and $\beta_2$, we write $\beta_1\sim \beta_2$ if they are conjugate or equivalent.

\begin{claim}\label{cl:braid}
$\sigma_{i(1)} \sigma_{i(2)} \dots \sigma_{i(n-1)} (\sigma_1 \sigma_2 \dots \sigma_{n-1})^{nrs}
\sim (\sigma_1\sigma_2 \dots \sigma_{n-1})^{nrs+1}$.
\end{claim}

\begin{proof}[Proof of Claim \ref{cl:braid}]
Put $F=(\sigma_1 \sigma_2 \dots \sigma_{n-1})^{nrs}$, which is central in the braid group.
First,  write $\sigma_{i(1)} \sigma_{i(2)} \dots \sigma_{i(n-1)} F=U_1 \sigma_1 U_2 F$,  where
$U_i$ is a word without $\sigma_1$, which is possibly empty.
Then $U_1\sigma_1 U_2 F \sim   \sigma_1 U_2F U_1\sim \sigma_1 U_2U_1F$.
Next, set $U_2U_1=V_1\sigma_2 V_2$, where $V_i$ is a (possibly, empty) word without $\sigma_1, \sigma_2$.
Note that $\sigma_1$ and $V_1$ commute.
Then 
\[
\sigma_1 U_2U_1F =\sigma_1 V_1\sigma_2 V_2 F \sim V_1\sigma_1 \sigma_2 V_2F\sim \sigma_1\sigma_2V_2FV_1 \sim \sigma_1\sigma_2 V_2V_1F.
\]
Repeating this procedure,  we have the conclusion.
\end{proof}

Thus the pattern $P$ is the closure of a braid $(\sigma_1\sigma_2\dots \sigma_{n-1})^{nrs+1}$.
This means that $k$ is a cable knot.
\end{proof}

\begin{remark}
Lee \cite[Question 1.2]{L3} asks whether  $T(p,q;r, s)$ is a cable knot, if it is a satellite knot under a condition that
$1<p<q$, $r\ne q$, $r$ is not a multiple of $p$, $1<r\le p+q$ and $s>0$.
Proposition \ref{prop:full}
gives a positive answer if the knot has tunnel number one, which is known to be true when $r\in \{2,3\}$ (see \cite{JL}).
\end{remark}

\begin{proposition}\label{prop:hyp}
If $p\ge 5$, then $K$ is hyperbolic.
\end{proposition}

\begin{proof}
First, $K=T(p,kp+1;2,1)$ is a torus knot if and only if
$p=2,3$ by \cite[Theorem 1.1]{L2}.
Hence we know that our knot is not a torus knot.

Assume that $K$ is a satellite knot for a contradiction.
We remark that $K$ has tunnel number one.
(A short arc at the extra full twist  gives an unknotting tunnel.)
Proposition \ref{prop:full} shows
that $K$ is the $(n,nrs+1)$-cable of $T(r,s)$.
Then $K=T(4,4m+1; 2,1)$ for some $m\ge 1$ by \cite{L3}.
This is a contradiction, because of  Lemma \ref{lem:alex} and $p\ge 5$.
Thus we have shown that $K$ is neither a torus knot nor a satellite knot, so
$K$ is hyperbolic.
\end{proof}

\section{Upsilon torsion function}\label{sec:upsilon-torsion}

In this section, we determine the Upsilon torsion function $\Upsilon^{\mathrm{Tor}}(K)$ of $K=T(p,kp+1;2,1)$.
Since $K$ is an L--space knot (Lemma \ref{lem:tunnel}), the full knot Floer complex $\mathrm{CFK}^\infty(K)$ is determined
by the Alexander polynomial (\cite{OS}). 
It has the form of staircase  diagram described by the gaps of Alexander polynomial.
If the gaps are given as a sequence $a_1,a_2,\dots, a_n$, then
the terms give the length of horizontal and vertical steps.
More precisely, let $g$ be the genus of $K$.
Start at the vertex $(0,g)$ on the coordinate plane.
Go right $a_1$ steps, and down $a_2$ steps, and so on. Finally, we reach $(g,0)$.
By the symmetry of the Alexander polynomial, the staircase inherits the symmetry along the line $y=x$.

We follow the process in \cite[Appendix]{AL}.
However, we assign a modified filtration level $\FL$ to each generator of the complex.
If a generator $x$ has the coordinate $(p,q)$, then $\FL(x)=tq+(2-t)p$.
In fact, for any  $t\in [0,2]$,
$\FL$ defines a real-valued function on $\mathrm{CFK}^\infty(K)$.
Then, for all $s\in \mathbb{R}$, $\mathcal{F}_s$ is spanned by all vectors
$x \in \mathrm{CFK}^\infty(K)$ such that $\FL(x)\le s$.
The collection $\{\mathcal{F}_s\}$ gives a filtration on $\mathrm{CFK}^\infty(K)$.  See \cite{L}.
(Remark that this filtration level is just the twice of that used in \cite{AL}.)
Since $\mathcal{F}_s\subset \mathcal{F}_u$ if $s\le u$, 
a generator $x_i\in \mathcal{F}_u$ can be added by $x_j\in \mathcal{F}_s$, without any change of the filtration level.
That, $\FL(x_i)=\FL(x_i+x_j)$.

For the staircase complex,
repeating a change of basis gradually splits the complex  into a single isolated generator
and separated arrows.
Then the value of the Upsilon torsion function is given as the maximum difference between  filtration levels among the arrows.

Since the Upsilon torsion function, defined on $[0,2]$, is symmetric along $t=1$, it suffices to consider the domain $[0,1]$.

As the simplest case, we demonstrate the process when $p=2$.

\begin{example}\label{ex:p=2}
Let $p=2$.
Then $K=T(2,2k+3)$ as mentioned before, and we show that
its Upsilon torsion function $\Upsilon^{\mathrm{Tor}}_K(t)=t \ (0\le t\le 1)$, independent of $k$.

By Corollary \ref{cor:gap}, the gaps of the exponents of the Alexander polynomial is $1,1,\dots, 1$ (repeated $2k+2$ times).
Hence the staircase diagram has the form as shown in Figure \ref{fig:p=2}, where
 $A_i$ has Maslov grading $0$, but $B_i$ has grading $1$, and each arrow has length one.

\begin{figure}[ht]
\centering
\includegraphics*[scale=0.7]{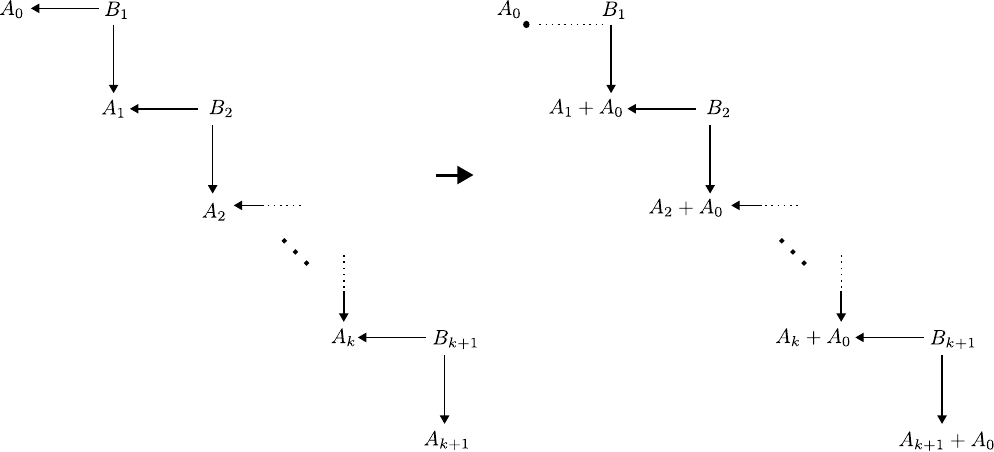}  
\caption{Left: The staircase diagram when $p=2$. The generator $A_i$ has grading $0$, but $B_i$ has grading $1$.
Each arrow has length one.
Right: By adding $A_0$ to $A_1,\dots, A_{k+1}$, the generator $A_0$ is isolated from the complex.}\label{fig:p=2}
\end{figure}

Each generator is assigned the filtration level $\FL$.
The difference between filtration levels among the generators is important.
We have $\FL(B_{i+1})-\FL(A_i)=2-t$ and $\FL(B_i)-\FL(A_i)=t$, because each arrow has length one.
Thus we have
\[
\FL(A_0)\le \FL(A_1)\le \dots \le \FL(A_{k+1}),
\]
where each equality occurs only when $t=1$.
Hence $A_0$ has the lowest filtration level among the generators with grading $0$.
Add $A_0$ to $A_1,\dots, A_{k+1}$.  Then the generator $A_0$ is isolated from the complex as shown in Figure \ref{fig:p=2}.
(Recall that we use $\mathbb{F}_2$ coefficients.)
In the remaining part of the complex,  $A_1+A_0$ is the lowest, since $\FL(A_i+A_0)=\FL(A_i)$ for $i=1,2,\dots,k+1$. 
To simplify the notation, we keep the same symbol $A_i$, instead of $A_i+A_0$, after this, if no confusion can arise.

Add $A_1$ to the other generators with grading $0$, except $A_0$.
Then the arrow $B_1\rightarrow A_1$ is split off from the complex.
Repeating this process leads to the decomposition of the original staircase into one isolated generator $A_0$ and 
$k+1$ vertical arrows.
For each arrow, the difference of filtration levels is equal to $t$, so the maximum difference is $t$ among the arrows.
This  shows $\Upsilon^{\mathrm{Tor}}_K(t)=t$.
\end{example}

\begin{theorem}\label{thm:upsilon-torsion}
Let $p\ge 4$.
The Upsilon torsion function $\Upsilon^{\mathrm{Tor}}_K(t)$ is given as
\[
\Upsilon_K^{{\rm Tor}}(t)=
\begin{cases}
(p-1)t & (0\le t \le \frac{2}{p})\\
2-t & (\frac{2}{p}\le t \le \frac{2}{p-2})\\
(p-3)t & (\frac{2}{p-2}\le t \le \frac{4}{p})\\
2m+(-m-1)t & (\frac{2m}{p}\le t \le \frac{2m}{p-1},\ m=2,\dots, \lfloor\frac{p-1}{2}\rfloor)\\
(p-2-m)t & (\frac{2m}{p-1}\le t\le \frac{2(m+1)}{p},\ m=2,\dots,\lfloor\frac{p}{2}\rfloor-1).
\end{cases}
\] 
In particular, $\Upsilon^{\mathrm{Tor}}_K(1)=\lfloor \frac{p-2}{2}\rfloor$.
\end{theorem}

\begin{proof}
Recall that the gaps are
\[
(1,p-1)^k,(1,1,1,p-3)^k,(1,1,2,p-4)^k,\dots, (1,1,p-3,1)^k,1,1,(p-1,1)^k
\]
by Corollary \ref{cor:gap}.
We name the generators of the staircase as in Figures \ref{Fig1}, \ref{Fig2} and  \ref{Fig3}.

\begin{figure}[ht]
\centering
\includegraphics*[scale=0.3]{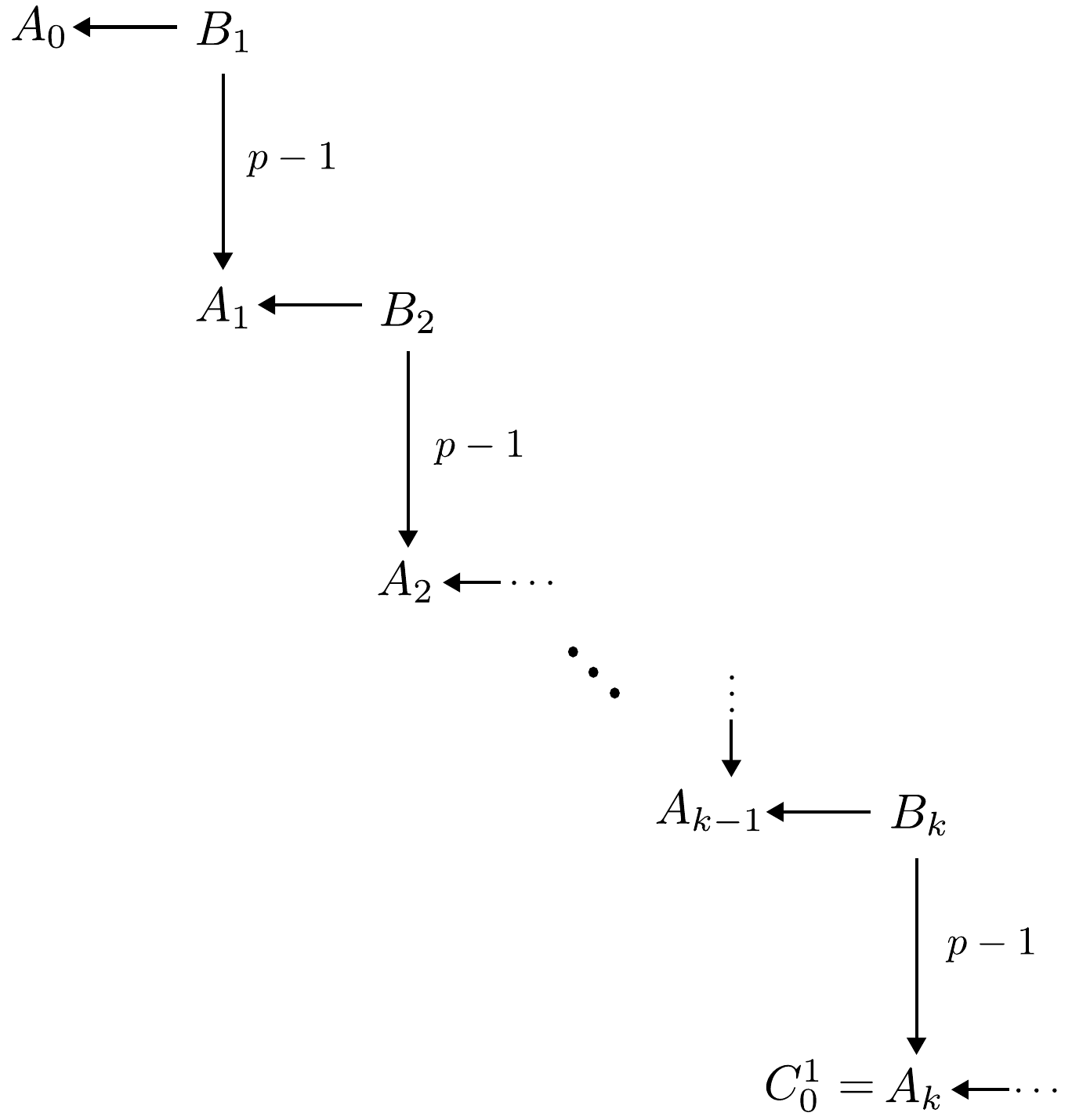}  
\caption{The first part corresponds to $(1,p-1)^k$. The generators $A_i$ have Maslov grading $0$, but
$B_i$ have $1$.
 The number $p-1$ next to each vertical arrow indicates the length.
Each horizontal arrow has length one.
Here, $C^1_0=A_k$. }\label{Fig1}
\end{figure}

\begin{figure}[ht]
\centering
\includegraphics*[scale=0.3]{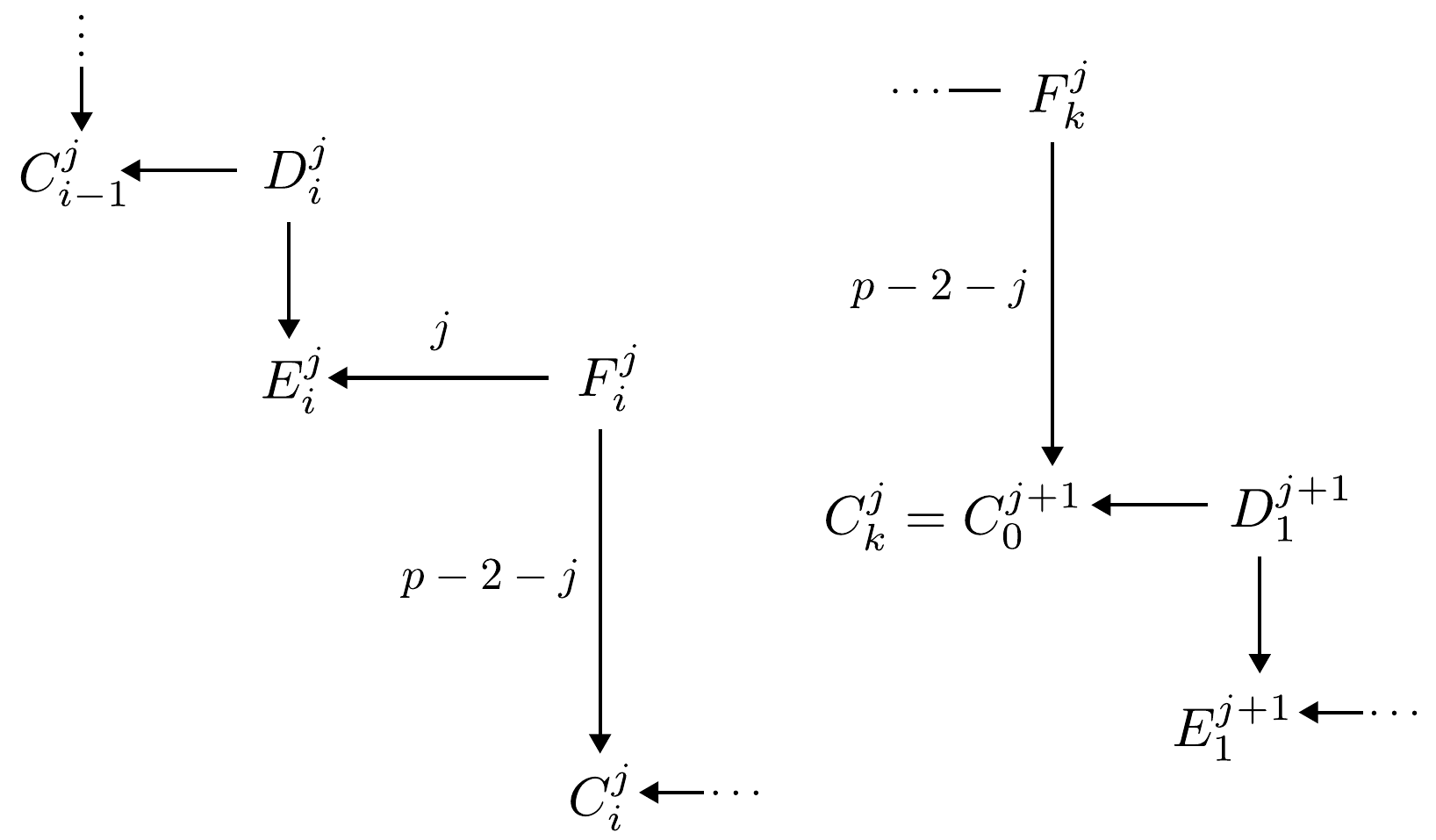}  
\caption{Left: The second part corresponds to $(1,1,j,p-2-j)^k\ (j=1,\ldots,p-3)$.
The generators $C_*^j$ and $E_*^j$ have Maslov grading $0$, but the others have $1$.
Right: This is a connecting part between $(1,1,j,p-2-j)$ and $(1,1,j+1, p-2-(j+1))$.}\label{Fig2}
\end{figure}

\begin{figure}[ht]
\centering
\includegraphics*[scale=0.3]{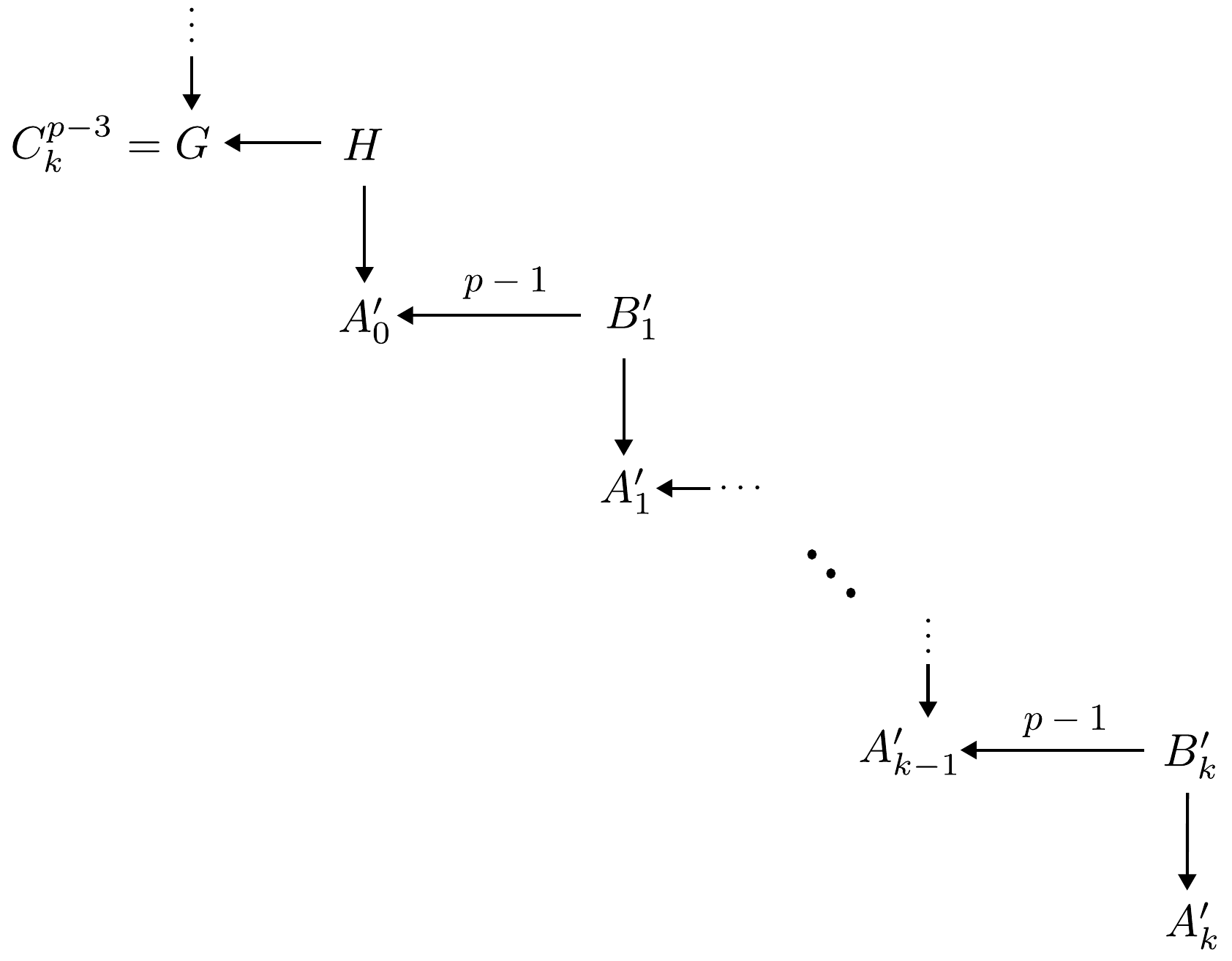}  
\caption{The last part corresponds to $1,1,(p-1,1)^k$.
The generators $G$ and $A_*'$ have Maslov grading $0$, but the others have $1$.
}\label{Fig3}
\end{figure}

In particular, we have 
 the difference between filtration levels of certain generators with Maslov grading 0
 as in Table \ref{table:filtration}.
 The argument is divided into 4 cases.


\begin{table}[htbp]
\renewcommand{\arraystretch}{1.2}
\begin{tabular}{l |l}
\hline
Difference & Indices\\
\hline
$\FL(A_i)-\FL(A_{i-1})=2-pt$ & $i=0,\ldots,k$\\
$\FL(E^j_i)-\FL(C^j_{i-1})=2-2t\ge 0$ & $i=1,\ldots,k;\ j=1,\ldots,p-3$\\
$\FL(C^j_i)-\FL(E^j_i)=(2-p)t+2j $& $i=1,\ldots,k;\ j=1,\ldots,p-3$\\
$\FL(C^j_i)-\FL(C^j_{i-1})=-pt+2(j+1)$ & $i=1,\ldots,k;\ j=1,\ldots,p-3$\\
$\FL(C^j_0)-\FL(C^{j-1}_0)=(-pt+2j)k$ & $j=2,\ldots,p-3$ \\
$\FL(A'_0)-\FL(G)=2-2t\ge 0$ & \\
$\FL(A'_i)-\FL(A'_{i-1})=-pt+2p-2>0$ & $i=1,\ldots,k$\\
\hline
\end{tabular}
\caption{Difference between filtration levels of the generators
with Maslov grading $0$.}
\label{table:filtration}
\end{table}

\medskip
\textbf{Case 1. $0\le t \le \frac{2}{p}$.}
Then any difference in Table \ref{table:filtration} is at least $0$.
Hence $A_0$ has the lowest filtration level among the generators with grading $0$, whose filtration levels
increase when we go to the right.

Exactly as in Example \ref{ex:p=2}, 
the staircase complex is decomposed  into a single isolated generator $A_0$
and separated vertical arrows $B_i\rightarrow A_i\ (i=1,2,\dots,k)$.
Hence the maximum difference of filtration levels on the arrows is $(p-1)t$.
This gives $\Upsilon^{\mathrm{Tor}}_K(t)=(p-1)t$ for $0\le t \le 2/p$.

\medskip
\textbf{Case 2. $\frac{2}{p}\le t \le \frac{4}{p}$.}
Then $\FL(A_0)\ge \FL(A_1)\ge \dots \ge \FL(A_k)$.
After $A_k$, the filtration levels increase among the generators with grading $0$, so $A_k$ is the lowest.
Add $A_k$ to the other generators with grading $0$.
Then $A_k$ will be isolated, and the complex splits into two parts.
We say that the first part, which starts at $A_0$ and ends at $B_k$, is \textit{N-shaped}, but the second, which starts at $D_1^1$ and ends at $A_k'$, is \textit{mirror N-shaped}.
In general,
if a ``zigzag'' complex starts and ends at horizontal arrows, then it is N-shaped.
If it starts and ends at vertical arrows, then it is mirror N-shaped.

For the first part, add $A_{k-1}$ to the others with grading $0$, which splits the arrow $A_{k-1} \leftarrow B_k$ off.
Repeat this as in Case 1.  Then the N-shaped complex is decomposed into separated horizontal arrows $A_{i-1}\leftarrow B_i\ (i=1,2,\dots,k)$, each of which has difference $2-t$.
The mirror N-shaped complex is also decomposed into vertical arrows similarly.
Thus the maximum difference among them is $(p-3)t$.

Compare $2-t$ and $(p-3)t$.
If $\frac{2}{p}\le t \le \frac{2}{p-2}$, then $2-t \ge (p-3)t$.
If $\frac{2}{p-2}\le t \le \frac{4}{p}$, then $2-t\le (p-3)t$.
Hence $\Upsilon^{\mathrm{Tor}}_K(t)=2-t$ for $\frac{2}{p}\le t\le \frac{2}{p-2}$, and $(p-3)t$  for $\frac{2}{p-2}\le t  \le\frac{4}{p}$.

\medskip
\textbf{Case 3. $\frac{2m}{p}\le t \le \frac{2m}{p-1}\ (m=2,\dots,\lfloor (p-1)/2\rfloor)$.}

From Table \ref{table:filtration}, we see that $C_k^{m-1}=C_0^m$ is the lowest among the generators with grading $0$.
Adding this to the others with grading $0$ decomposes the complex into one isolated generator $C_0^m$,
the N-shaped one between $A_0$ and $F_k^{m-1}$ and the mirror N-shaped one between $D_1^m$ and $A_k'$.

As before,  the mirror N-shaped complex can be decomposed into vertical arrows.
The longest arrows has length $(p-2-m)t$.

\begin{figure}[ht]
\centering
\includegraphics*[scale=0.4]{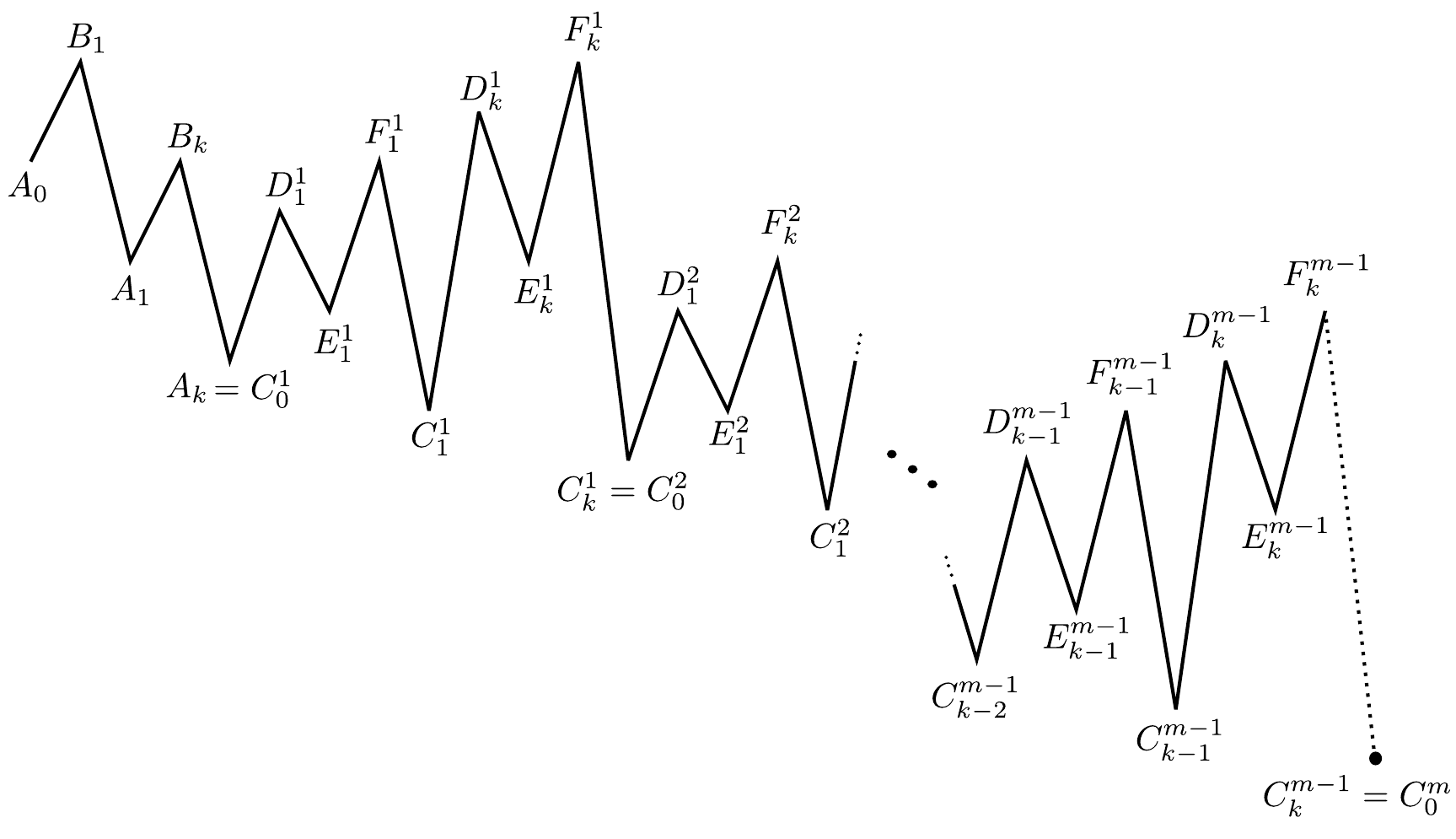}  
\caption{The N-shaped complex between $A_0$ and $F_k^{m-1}$ after isolating the lowest vertex $C_0^m$, where $k=2$.
The height indicates  the filtration level of each generator.
As before, we keep the same notation for generators after a change of basis.
}\label{fig:case3}
\end{figure}

The N-shaped complex is described in Figure \ref{fig:case3}.
We have
\[
\begin{split}
\FL(A_0)\ge \FL(A_1)\ge \cdots &\ge \FL(A_k=C_0^1)\ge \FL(C_1^1)\ge \dots \ge \FL(C_k^1=C_0^2)\\
&\ge \FL(C_1^2)\ge \dots \ge \FL(C_{k-1}^{m-1})
\end{split}
\]
and $\FL(E_i^j)\ge \FL(C_{i-1}^j)\ (i=1,2,\dots,k; j=1, 2\dots, m-1 )$.

Hence $C_{k-1}^{m-1}$ is the lowest.
Adding this to the others with grading $0$ on the left splits an N-shaped complex 
$C_{k-1}^{m-1} \leftarrow D_k^{m-1} \rightarrow E_k^{m-1} \leftarrow F_k^{m-1}$
off.
For the remaining part, the lowest is $C_{k-2}^{m-1}$.
Again, adding this to the others with grading $0$ on the left splits an N-shaped complex
$C_{k-2}^{m-1} \leftarrow D_{k-1}^{m-1} \rightarrow E_{k-1}^{m-1} \leftarrow F_{k-1}^{m-1}$
off.
Repeat this, then we obtain  an N-shaped complex between $A_0$ and $B_k$, and
N-shaped complexes 
$C_{i-1}^j \leftarrow D_{i}^{j} \rightarrow E_{i}^{j} \leftarrow F_{i}^{j} \ (i=1,\dots, k; j=1,\dots,m-1)$.

For the former, the process as in Case 2 yields separated horizontal arrows, each of which has difference $2-t$.
Let us consider the latter N-shaped ones.
Since $\FL(F_i^j)-\FL(D_i^j)=-t+(2-t)j\ge 0$, add $D_i^j$ to $F_i^j$.
After that, add $C_{i-1}^j$ to $E_i^j$.
As shown in Figure \ref{fig:case3N}, 
this change of basis decomposes the complex into a pair of arrows.
One has difference $2-t$, and the other has difference $\FL(F_i^j)-\FL(C_{i^-1}^j)=2(j+1)+(-j-2)t$.
Note $2-t\le 2(j+1)+(-j-2)t$.
Furthermore, for $1\le j\le m-1$, $j=m-1$ attains the maximum value, $2m+(-m-1)t$.

\begin{figure}[ht]
\centering
\includegraphics*[scale=0.8]{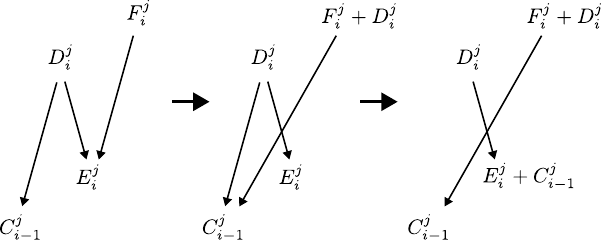}  
\caption{A change of basis for an N-shaped complex.
Add $D_i^j$ to $F_i^j$, and $C_{i-1}^j$ to $E_i^j$.}\label{fig:case3N}
\end{figure}

Hence we need to compare the values $(p-2-m)t$ and $2m+(-m-1)t$.
Since $2m+(-m-1)t \ge (p-2-m)t$,
we have $\Upsilon^{\mathrm{Tor}}_K(t)=2m+(-m-1)t$ for this case.

\medskip
\textbf{Case 4. $\frac{2m}{p-1}\le t \le \frac{2(m+1)}{p}\ (m=2,\dots,\lfloor p/2\rfloor-1)$.}

As in Case 3, $C_k^{m-1}=C_0^m$ is the lowest.
So, adding this to the others with grading $0$ decomposes the complex into one isolated generator $C_0^m$,
the N-shaped one and the mirror N-shaped one, again.

For the N-shaped complex, the situation is the same as in Case 3.
Thus we have an arrow with maximum difference $2m+(-m-1)t$ from this N-shaped complex.

However, we need to handle the mirror N-shaped complex differently now.

First, consider the case where $\frac{2m}{p-1}\le t\le \frac{2m}{p-2}$.
Then the filtration levels of  the generators with grading $0$ increase as going to the right.
So, as in Case 3, this part can be decomposed into vertical arrows, and the longest has length $(p-2-m)t$.

Second, consider the case where $\frac{2m}{p-2}\le t\le \frac{2(m+1)}{p}$.
Then $\FL(E_i^m)\ge \FL(C_i^m)\ge \FL(C_{i-1}^m)\ (i=1,2,\dots,k)$, but
the filtration levels of
the remaining generators with grading $0$, $C_0^{m+1},E_1^{m+1},C_1^{m+1},\dots,G,A_0',\dots A_k'$,
increase as going to the right.
See Figure \ref{fig:case4}.

\begin{figure}[ht]
\centering
\includegraphics*[scale=0.35]{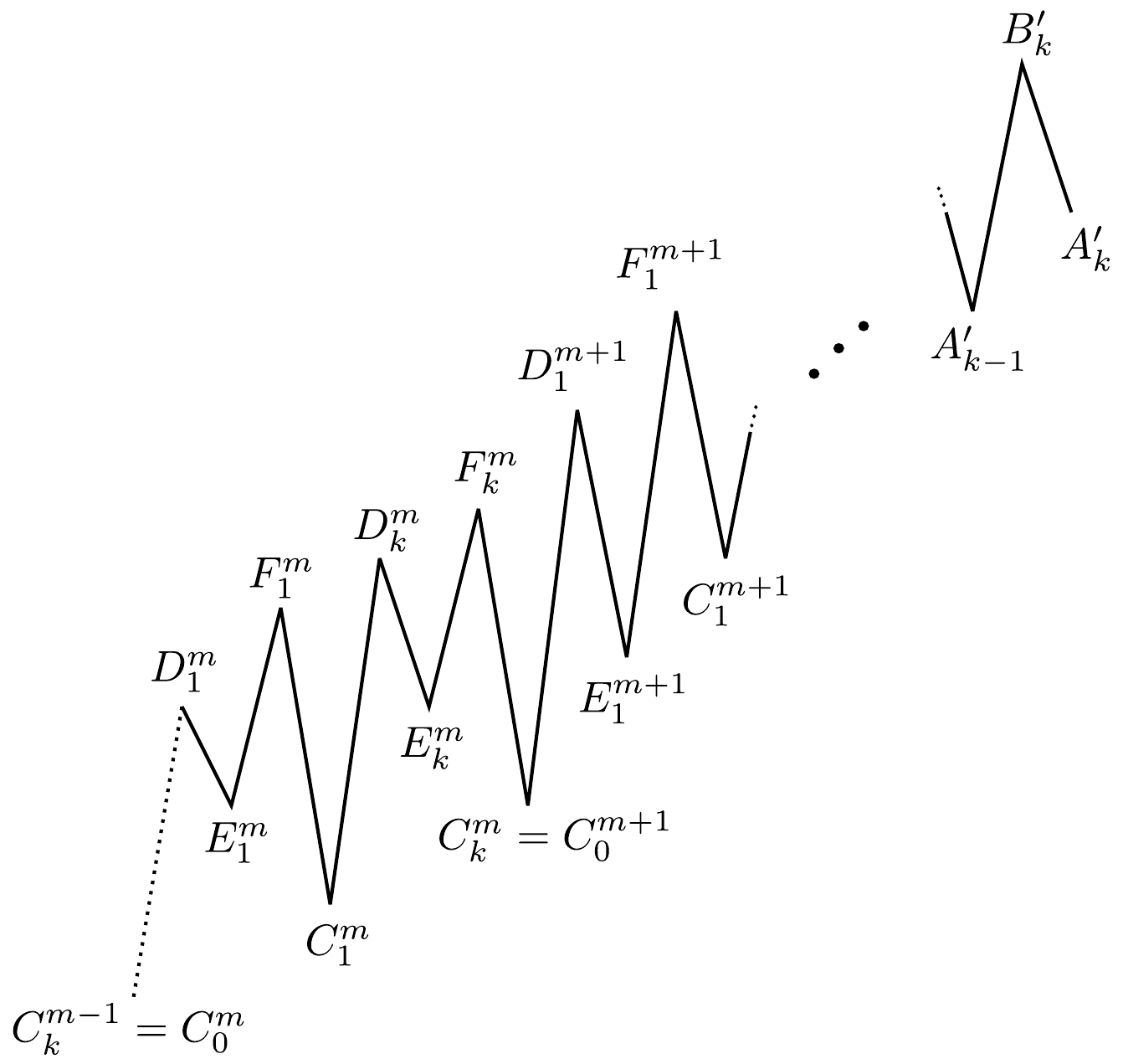}  
\caption{The mirror N-shaped complex when $\frac{2m}{p-2}\le t\le \frac{2(m+1)}{p}$, where $k=2$.
The generator $C_1^m$ is the lowest.}\label{fig:case4}
\end{figure}

Here, $C_1^m$ is the lowest.
Adding this to the others with grading $0$ on the right splits a mirror N-shaped complex 
$D_1^m\rightarrow E_1^m \leftarrow F_1^m \rightarrow C_1^m$ off.
Then $C_2^m$ is the lowest in the remaining part.
Repeating this yields mirror N-shaped complexes $D_i^m\rightarrow E_i^m \leftarrow F_i^m \rightarrow C_i^m\ (i=1,2,\dots,k)$,
and one more mirror N-shaped one between $D_1^{m+1}$ and $A_k'$.
For the last one, the previous process gives vertical arrows.
For each mirror N-complex $D_i^m\rightarrow E_i^m \leftarrow F_i^m \rightarrow C_i^m$,
we remark $\FL(F_i^m)-\FL(D_i^m)=(-m-1)t+2m\ge 0$.
Hence adding $D_i^m$ to $F_i^m$ yields a pair of vertical arrows as shown in Figure \ref{fig:case4N}.
Thus we have only vertical arrows, whose longest length is $(p-2-m)t$.

\begin{figure}[ht]
\centering
\includegraphics*[scale=0.8]{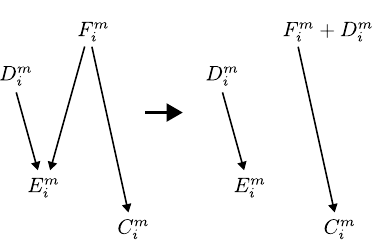}  
\caption{A change of basis for a mirror N-shaped complex.
Adding $D_i^m$ to $F_i^m$ yields a pair of vertical arrows.}\label{fig:case4N}
\end{figure}

Finally, compare $2m+(-m-1)t$ and $(p-2-m)t$.
Since $\frac{2m}{p-2}\le t\le \frac{2(m+1)}{p}$, the latter is bigger.
Then $\Upsilon^{\mathrm{Tor}}_K(t)=(p-2-m)t$ for this case.
\end{proof}

\begin{example}
When $p=6$,
\[
\Upsilon^{\mathrm{Tor}}_K(t)=
\begin{cases}
5t & (0\le t\le \frac{1}{3}) \\
2-t & (\frac{1}{3}\le t\le \frac{1}{2}) \\
3t & (\frac{1}{2}\le t\le \frac{2}{3}) \\
4-3t & (\frac{2}{3}\le t\le \frac{4}{5}) \\
2t & (\frac{4}{5}\le t\le 1).
\end{cases}
\]
See Figure \ref{fig:p=6}.
\end{example}

\begin{figure}[ht]
\centering
\includegraphics*[scale=0.7]{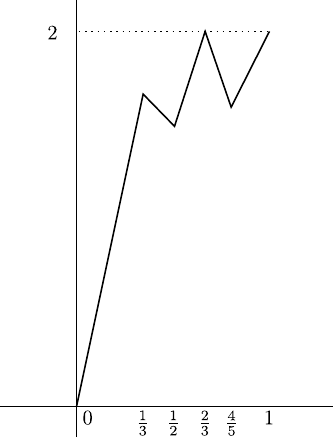}  
\caption{The Upsilon torsion function $\Upsilon^{\mathrm{Tor}}_K(t)$ 
of $K=T(6,6k+1;2;1)$.
Then $\Upsilon^{\mathrm{Tor}}_K(1)=2$.}\label{fig:p=6}
\end{figure}

\section{Torsion order}

We are ready to prove Theorem \ref{thm:main}.

\begin{proof}[Proof of Theorem \ref{thm:main}]
By \cite{AL}, $\Upsilon'^{\mathrm{Tor}}_K(0)=\mathrm{Ord}(K)$ and
$\Upsilon^{\mathrm{Tor}}_K(1)=\mathrm{Ord}'(K)$.
Thus Theorem \ref{thm:upsilon-torsion} immediately gives
$\mathrm{Ord}(K)=p-1$ and $\mathrm{Ord}'(K)=\lfloor (p-2)/2 \rfloor$ when $p\ge 4$.

When $p\in \{2,3\}$, $K$ is a torus knot, and
$\mathrm{Ord}(K)$ is equal to the longest gap in the exponents of the Alexander
polynomial by \cite[Lemma 5.1]{JMZ}.
Hence it is $p-1$ by Corollary \ref{cor:gap}.
(Indeed, the latter argument proves $\mathrm{Ord}(K)=p-1$ for any $p\ge 2$.)
\end{proof}

\begin{proof}[Proof of Corollary \ref{cor:main}]
By Proposition \ref{prop:hyp}, the twisted torus knot $K=T(p,kp+1;2,1)$  is hyperbolic if $p\ge 5$.
Since $K$ has genus $(kp^2-kp+2)/2$,
distinct choices of $k$, with a fixed $p$, give distinct knots.

Set $K_2=K$ with $p=2N+3\ge 5$.  Then $K_2$ is hyperbolic
and $\mathrm{Ord}'(K_2)=\lfloor (p-2)/2\rfloor=N$.

If $N\ge 4$, then set $K_1=K$ with $p=N+1$.
Then $K_1$ is hyperbolic and $\mathrm{Ord}(K_1)=p-1=N$.

To complete the proof,
we need infinitely many hyperbolic knots $K_1$ whose
$\mathrm{Ord}(K_1)$ takes each of the values $1,2, 3$.

\begin{enumerate}
\item
By \cite[Corollary 1.8]{JMZ}, $\mathrm{Ord}(L)\le b(L)-1$ for any knot $L$.
Hence if $K_1$ is a hyperbolic $2$-bridge knot, then $\mathrm{Ord}(K_1)=1$.
\item
Let $K_1=T(3,4;2,s)$ with $s\ge 2$.
Then  $K_1$ is an L--space knot (\cite{V}),  and twist positive in the sense of \cite{KM}.
In the proof of \cite[Theorem 1.3]{KM}, they show
that $\mathrm{Ord}(K_1)=2$.
By \cite{L1,L2}, $K_1$ is hyperbolic.
Since $K_1$ has genus $s+3$, distinct choices of $s$ give distinct knots.
\item
Finally, there are infinitely many hyperbolic L--space knots  $\{k_n\}$, defined in \cite[Section 2]{BK}, 
with $\mathrm{Ord}(k_n)=3$.  (See \cite[Proposition 5.1]{KM}.)
\end{enumerate}
\end{proof}

\begin{proof}[Proof of Corollary \ref{cor:main2}]
Let $K=T(p,kp+1;2,1)$ with $p\ge 5$.
Then $K$ is hyperbolic by Proposition \ref{prop:hyp}.
After fixing $p$,
Theorem \ref{thm:upsilon-torsion} shows that
the Upsilon torsion function does not depend on $k$.
\end{proof}


\bibliographystyle{amsplain}

\end{document}